\documentclass[a4paper]{article}
\usepackage[utf8]{inputenc}

\usepackage{hyperref}

\usepackage{tikz}
\usetikzlibrary{calc}
\usepackage{amsthm}
\usepackage{amsopn}
\usepackage{amsmath}
\usepackage{amssymb}
\usepackage{mathrsfs}

\usepackage{color}

\theoremstyle{plain}
\newtheorem{thm}{Theorem}

\newtheorem{lemma}[thm]{Lemma}

\newtheorem{question}[thm]{Question}

\theoremstyle{definition}

\newtheorem{claim}{Claim}

\newcommand*{\myproofname}{Proof}
\newenvironment{claimproof}[1][\myproofname]{\begin{proof}[#1]}{\end{proof}}

\title{Erd\H{o}s-P\'osa from ball packing}

\author{
Wouter Cames van Batenburg
\thanks{D\'epartement d'Informatique, 
Universit\'e Libre de Bruxelles, 
Brussels, Belgium.
Email: \protect\href{mailto:wcamesva@ulb.ac.be}{\protect\nolinkurl{wcamesva@ulb.ac.be}}, \protect\href{mailto:gjoret@ulb.ac.be}{\protect\nolinkurl{gjoret@ulb.ac.be}}. Supported by an ARC grant from the Wallonia-Brussels Federation of Belgium.}
\and
Gwena\"el Joret
\footnotemark[1]
\and
Arthur Ulmer
\thanks{Institut f\"ur Optimierung und Operations Research, Universit\"at Ulm, Germany. Email: \protect\href{mailto:arthur.ulmer@uni-ulm.de}{\protect\nolinkurl{arthur.ulmer@uni-ulm.de}}. Supported by DFG, grant no.\  BR 5449/1-1.}
}

\date{\today}

\begin{document}

\maketitle

\begin{abstract}
A classic theorem of Erd\H{o}s and P\'osa (1965) states that every graph has either $k$ vertex-disjoint cycles or a set of $O(k \log k)$ vertices meeting all its cycles. 
While the standard proof revolves around finding a large `frame' in the graph (a subdivision of a large cubic graph), an alternative way of proving this theorem is to use a ball packing argument of K\"uhn and Osthus (2003) and Diestel and Rempel (2005). 
In this paper, we argue that the latter approach is particularly well suited for studying edge variants of the Erd\H{o}s-P\'osa theorem.

As an illustration, we give a short proof of a theorem of Bruhn, Heinlein, and Joos (2019), that cycles of length at least $\ell$ have the so-called edge-Erd\H{o}s-P\'osa property. More precisely, we show that every graph $G$ either contains $k$ edge-disjoint cycles of length at least $\ell$ or an edge set $F$ of size $O(k\ell \cdot \log (k\ell))$ such that $G-F$ has no cycle of length at least $\ell$. For fixed $\ell$, this improves on the previously best known bound of $O(k^2 \log k +k\ell)$.
\end{abstract}

\section{Introduction}

By a classic theorem of Erd\H{o}s and P\'osa~\cite{Erdos1965independent}, every graph contains either $k$ vertex-disjoint cycles or a set of $O(k \log k)$ vertices that meets every cycle. This result has sparked a vast body of generalizations and variations regarding the so-called Erd\H{o}s-P\'osa property. 
A class of graphs $\mathcal{H}$ is said to have the \emph{vertex-Erd\H{o}s-P\'osa property} (resp.\ \emph{edge-Erd\H{o}s-P\'osa property}) if there exists a \emph{bounding function} $f(k): \mathbb{N}\rightarrow \mathbb{R}_{\geq 0}$ such that for all $k\in \mathbb{N}$, every graph $G$ contains either $k$ vertex-disjoint (resp.\ edge-disjoint) subgraphs in $\mathcal{H}$, or a set $F$ of at most $f(k)$ vertices (resp.\ edges) such that $G-F$ has no subgraph in $\mathcal{H}$. Henceforth, we will abbreviate the name of this property to \emph{vertex-EP-property} (resp.\ \emph{edge-EP-property}). 

The result of Erd\H{o}s and P\'osa implies that the class of cycles in fact has both the vertex and the edge-EP property (cp.~\cite{Diestel}). Despite this, most research since then has focused on the vertex-EP property.
For instance, by a seminal result of Robertson and Seymour~\cite{robertson1986graph}, the class of $J$-expansions has the vertex-EP property if and only if $J$ is a planar graph. Here, given a graph $J$, a \emph{$J$-expansion} is a graph that contains $J$ as a minor. Recently~\cite{CHJR19} it has been established that the expansions of any given planar graph $J$ have bounding function $O (k \log k)$, where the hidden constant depends on the size of $J$. 
This is best possible when $J$ has a cycle (if $J$ is a forest then a $O(k)$ bound holds~\cite{FJW2013}). 
Note that the Erd\H{o}s-P\'osa theorem can be viewed as a special case of this result, since each cycle is a $C_3$-expansion and every $C_3$-expansion contains a cycle as a subgraph. 
 
While the vertex-EP property of expansions is now quite well understood, the edge-EP property has not been thoroughly investigated. Interest in the edge-EP property is more recent, starting with the works of Birmelé, Bondy, and Reed~\cite{BBR07} and that of Sau, Raymond, and Thilikos~\cite{RST16}. 
Since then, a series of papers have demonstrated the edge-EP property for $\theta_r$-expansions~\cite{CRST18} (where $\theta_r$ is the multigraph consisting of two vertices joined by $r$ parallel edges), for \emph{long cycles}~\cite{BHJ19} (i.e.: the cycles of length at least some given constant $\ell$) and for $K_4$-expansions~\cite{BH18}. There has also been interest in the edge-EP property for labeled graphs~\cite{HU19, B19} and directed graphs~\cite{HU18}. From this account, one might gain the impression that the vertex-EP property and edge-EP property are essentially the same, and indeed sometimes the edge-EP property can be derived directly from its vertex variant~\cite{PW12}. However, in~\cite{BHJ18_counterexamples}, this intuition has been disproved in a strong sense: there exist planar graphs $J$ such that $J$-expansions do \emph{not} have the edge-EP property. In particular, this is the case when $J$ is a sufficiently large ladder or a subcubic tree with sufficiently large pathwidth. 
Finding a characterization for the planar graphs $J$ such that $J$-expansions have the edge-EP property remains an intriguing open problem. \\

In this paper, we develop a ball packing technique to establish the edge-EP property of some graph classes. 
As a warm-up, we first show how it yields a short self-contained proof of the Erd\H{o}s-P\'osa theorem for cycles, along the lines of Raymond and Thilikos~\cite{Raymond2017}; see Section~\ref{sec:classicEP}. 

After that, we apply the technique to obtain a relatively short proof that long cycles have the edge-EP property, with an improved bounding function. 

It has been long known that long cycles (of length at least $\ell$) have the \emph{vertex}-EP property; this follows for instance from~\cite{robertson1986graph}. 
The first polynomial bounding function is due to Birmelé, Bondy, and Reed~\cite{BBR07}, who obtained an $O(k^2\ell)$ bounding function.  
Then Fiorini and Herinckx~\cite{FH14} improved this to an $O(k\ell \log k)$ bounding function. Using a different approach, this has been refined by Mousset, Noever, and Škori\'c, and Weissenberger~\cite{MNSW17} to an $O(k \log k + k\ell)$ bounding function, which is best possible.

After this, Bruhn, Heinlein and Joos~\cite{BHJ19} established that cycles of length at least $\ell$ also have the \emph{edge}-EP property. They obtained an $O(k^2 \log k+ k\ell)$  bounding  function. Among other ingredients, they used the vertex-EP property, a reduction which roughly costs them a multiplicative factor $k$. In contrast, we do not require this reduction step; our technique is self-contained and can simultaneously yield the vertex and edge-EP property. This is the main reason we are able to prove an $O(k\ell \cdot \log (k\ell))$ bound:  

\begin{thm}\label{thm:longcycles}
For every integer $\ell\geq 2$, the edge-EP property holds for cycles of length at least $\ell$, with bounding function $8k(\ell-1)(\lceil\log_2(k\ell)\rceil+1)$.
\end{thm}
We remark that for fixed $\ell$, an $O(k \log k)$ bound is best possible. 
On the other hand, if we let $\ell$ vary, then there might still be some room for improvement. 
It can be shown~\cite{BHJ19} that it is not possible to improve the bound of Theorem~\ref{thm:longcycles} below $O(k \log k + k\ell)$ just as in the vertex case, but we are not aware of any better lower bound. \\

The paper is organised as follows. 
First, a sketch of the general approach is given in Section~\ref{sec:sketch}. 
Then we use it in Section~\ref{sec:classic_EP} to give an alternative proof of the Erd\H{o}s-P\'osa theorem, and in Section~\ref{sec:long_cycles} to prove Theorem~\ref{thm:longcycles}. 
Finally, we conclude the paper with some open problems (Section~\ref{sec:open_problems}).

\section{A sketch of the proof and beyond}
\label{sec:sketch}

We now provide a sketch of the general proof scheme used for Theorems~\ref{thm:longcycles} and~\ref{thm:vertexcycle}, and how this might be useful to prove the edge-EP property for other graph classes as well. \\

Suppose we are given a fixed graph $J$ and we want to prove the edge-EP property for the class of $J$-expansions, with an $O(k \log k)$ bounding function. Let $k \in \mathbb{N}$ and let $G$ be a graph. The first step is to apply induction on the number of edges of $G$. If $G$ has a subgraph $H$ which is a small $J$-expansion, of size at most $O (\log k)$, then one can apply a straightforward induction to the graph $G-E(H)$. Hence we may assume that each $J$-expansion in $G$ is large, of size $\Omega(\log k)$. 
Similarly, if we can apply a reduction operation on $G$, then we are done by induction, so we may assume that $G$ is {\em reduced}. 
Here, the reduction operations depend on $J$ (for instance, removing vertices of degree-$1$ and suppressing vertices of degree-$2$ if $J$ is $C_3$).  
Finally, we partition the vertex set $V(G)$ of our reduced graph $G$ into sets that induce connected subgraphs of diameter proportional to $\log k$; in fact these subgraphs can be chosen to approximately equal balls of some uniform radius $r=O(\log k)$. For simplicity, we call them \emph{balls} in the remainder of this proof sketch. Using that small $J$-expansions do not occur as a subgraph, we then need to derive that each ball grows exponentially, in the sense that the number of vertices on its boundary is exponential in $r$. 
This ensures that the boundary of each ball has size $\Omega(k^{1+\epsilon})$, for some $\epsilon$ depending on the size of $J$ but independent of $k$. We also need that sufficiently many vertices on the boundary have a neighbour in some other ball, and that between any two balls there are only `few' edges. This will ensure that every ball sends an edge to $\Omega(k^{1+\epsilon})$ distinct other balls. 
Once this has been established, we can contract each ball to a vertex, thus obtaining a minor of $G$ with 
 minimum degree $\Omega(k^{1+\epsilon })$ after removing parallel edges.
 This in turn implies that $G$ has a large clique minor (a clique of order at least $k \cdot |V(J)|$) and hence $G$ contains $k$ vertex-disjoint (and thus also edge-disjoint) $J$-expansions, as desired. This concludes the sketch of the general method.\\
  
Finding adequate reduction operations is of course problem dependent (i.e.\ it depends on $J$). 
The main challenge with this method seems to be proving that the ball boundaries are uniformly increasing as a function of their radius once the graph is reduced; ideally exponentially in the radius. 
This is not difficult in the case of $J=C_3$, because then one can deduce that the balls must induce trees where each internal vertex has degree at least $3$, see Section~\ref{sec:classicEP}. In the case where $J$ is some larger cycle $C_\ell$, the balls do not necessarily induce trees, but it is still possible to reduce to the case where all balls globally possess some tree-like structure, yielding the desired growth behaviour. This step is the bulk of the proof of Theorem~\ref{thm:longcycles}. \\

We wish to reiterate that \emph{if} this method works for a given graph $J$, then essentially the same proof also yields the vertex-EP property. Moreover, if the balls exhibit exponential growth, then this method yields the optimal $O(k \log k)$ bound. If the balls grow slower, e.g. polynomially in the radius, then the method will still yield the edge-EP property, albeit with a larger bounding function.\\

\section{Classic Erd\H{o}s-P\'osa}\label{sec:classicEP}
\label{sec:classic_EP}

In this section, we present the aforementioned proof of the Erd\H{o}s-P\'osa theorem for cycles, as a warm-up for the proofs presented in Section~\ref{sec:long_cycles}. It is along the lines of the proof given in~\cite{Raymond2017}. Our main tool is the following powerful lemma of K\"uhn and Osthus~\cite{KO03}, see also Diestel and Rempel~\cite{DR05}.  Its short proof is included for completeness.

Given two vertices $u,v$ in a graph $G$, we let $d_G(u,v)$ denote the distance between $u$ and $v$ in $G$. 
\begin{lemma}[K\"uhn and Osthus~\cite{KO03}]\label{lem:mindegreeminor}
Let $m\geq 0$ and $d\geq 3$ be integers. Then every graph $G$ of girth at least $8m+3$ and minimum degree $d$ contains a minor of minimum degree at least $d(d-1)^m$.
\end{lemma}
\begin{proof}
This proof is taken over literally from~\cite{KO03}. 
The lemma is obvious if $m=0$, so let us assume $m \geq 1$. 
For each $x\in V(G)$ and integer $r\geq 0$, let $B_r(x):= \left\{ u \in V(G) \mid d_G(x,u) \leq r \right\}$ denote the ball of radius $r$ centered at $x$. Let $X$ be a maximal set of vertices of $G$ that have pairwise distance at least $2m+1$ from each other.  Thus, for distinct $x,y\in X$, the balls $B_m(x)$ and $B_m(y)$ are disjoint. Extend the $(B_m(x))_{x\in X}$ to disjoint connected subgraphs of $G$ by first adding each vertex at distance $m+1$ from $X$ to one of the $B_m(x)$ to which it is adjacent. Then add each vertex at distance $m+2$ from $X$ to one of the subgraphs constructed in the previous step to which it is adjacent. Continue like this until each vertex of $G$ is contained in one of the constructed subgraphs and denote the subgraph obtained from $B_m(x)$ in this way by $T(x)$. The choice of $X$ implies that each vertex of $G$ has distance at most $2m$ from $X$, so each vertex of $T(x)$ has distance at most $2m$ from $x$ in $T(x)$. Therefore, as $G$ has girth at least $4m+2$, each $T(x)$ is an induced subtree of $G$. In particular, $B_m(x)$ is a tree in which every vertex that is not a leaf has degree at least $d$ and in which every  leaf has distance $m$ from $x$. So $B_m(x)$ and $T(x)$ have at least $d(d-1)^{m-1}$ leaves. Hence $T(x)$ sends at least $d(d-1)^m$ edges to vertices outside $T(x)$. Because $G$ has girth at least $8m+3$, two distinct trees $T(x), T(y)$ are joined by at most one edge. Thus the graph obtained from $G$ by contracting the trees $(T(x))_{x \in X}$ has minimum degree at least $d(d-1)^m $, as desired.
\end{proof}

 We will also need the following lemma.

\begin{lemma}\label{lem:mindegreepacking}
Let $k$ be a nonnegative integer and let $G$ be a graph of minimum degree at least $3k$. Then $G$ contains $k$ vertex-disjoint cycles.
\end{lemma}
\begin{proof}
We proceed by induction on $k$. The statement is clearly true for $k=0$. Suppose $k\geq 1$. Since $G$ has minimum degree at least $3$, it has a cycle. Let $C$ be a shortest cycle in $G$. Every vertex in $V(G)\backslash V(C)$ has at most three neighbours in $V(C)$, since otherwise there would be a shorter cycle. Thus the minimum degree of $G-V(C)$ is at least $3(k-1)$ and so, by induction, there exists a collection of at least $k-1$ vertex-disjoint cycles in $G-V(C)$. Adding $C$ to this collection, we deduce that $G$ contains $k$ vertex-disjoint cycles.
\end{proof}

Now everything is ready for the proof. We provide it for the vertex-version and then remark which small changes need to be made for the edge-version. We define $g(0):=0$ and $g(k):=8 \lceil \log_2(k) \rceil+2$, for every positive integer $k$.

\begin{thm}\label{thm:vertexcycle}
For every graph $G$ and every nonnegative integer $k$, either $G$ contains at least $k$ vertex-disjoint cycles, or there is a set $F\subseteq V(G)$ of size at most $k\cdot g(k)$ such that $G-F$ has no cycle. 
\end{thm}
\begin{proof}
We proceed by induction on $|V(G)|$, the base case $|V(G)|=0$ being clearly true for any $k$. Let $k\geq 1$ and let $G$ be a non-empty graph. Suppose first that $G$ has girth at most $g(k)$. Let $C$ be a shortest cycle in $G$ and consider the graph $G':=G-V(C)$. Note that $|V(C)|\leq g(k)$. By induction, either $G'$ contains $k-1$ vertex-disjoint cycles or there exists a set $F'\subseteq V(G')$ of size at most $(k-1)\cdot g(k-1)$ such that $G'-F'$ has no cycles. In the former case we can add $C$ to the collection to obtain $k$ vertex-disjoint cycles, as desired. In the latter case $G-F$ has no cycles, where $F:= F' \cup V(C)$ has size at most $|F'|+|V(C)|\leq  (k-1)\cdot g(k-1) + g(k) \leq  k \cdot g(k)$, as desired. Thus we may assume from now on that $G$ has girth at least $g(k)+1$.

If $G$ has a vertex $v$ of degree at most $1$, then no cycle visits $v$, so we are done by induction applied to $G-v$. If $G$ has a degree-$2$ vertex $v$ with neighbour $u$, then (since $G$ is triangle-free) we are done by induction applied to the graph obtained by contracting the edge $uv$.

Thus we may assume that $G$ has minimum degree at least $3$ and hence, by Lemma~\ref{lem:mindegreeminor}, $G$ has a minor $G^*$ of minimum degree at least  $3\cdot 2^{(g(k)-2)/8} \geq 3k$. Therefore Lemma~\ref{lem:mindegreepacking} yields $k$ vertex-disjoint cycles in $G^*$. 
Lifting these cycles to $G$ yields $k$ vertex-disjoint cycles in $G$, as desired. 
\end{proof}

The proof of the edge-variant of Theorem~\ref{thm:vertexcycle} is verbatim the same, apart from the canonical replacements of vertex sets with edge sets. In particular, in this case the induction goes on the number of edges. At the end of the proof, we still obtain $k$ \emph{vertex}-disjoint cycles in $G$; note that they are edge-disjoint as well.

\section{Edge-Erd\H{o}s-P\'osa for long cycles}
\label{sec:long_cycles}

In this section we prove the edge-Erd\H{o}s-P\'osa theorem for long cycles.

First we prove two auxiliary lemmas. The first of these lemmas provides a minimum degree condition for finding many vertex-disjoint long cycles:

\begin{lemma}\label{lem:disjointlongcycles}
Let $k \geq 1$ and  $\ell\geq 3$ be integers and let $G$ be a graph of minimum degree at least $k\ell-1$. Then $G$ contains $k$ vertex-disjoint cycles that are each of length at least $\ell$.
\end{lemma}
\begin{proof}
We proceed by induction on $k$. Starting with the base case $k=1$, let $G$ be a graph of minimum degree at least $\ell-1$. Then a maximal path $P$ in $G$ has at least $\ell$ vertices. Since all neighbours of the first vertex of $P$ are on $P$, we know that $G$ has a cycle of length at least $\ell$, as desired. 
Now let $k>1$ and assume that the lemma holds true for all values smaller than $k$. Let $G$ be a graph of minimum degree at least $k\ell-1$ and let $C$ be a shortest cycle of length at least $\ell$. Suppose for a contradiction that some $v \in V(G)\backslash V(C)$ has at least $\ell+1$ neighbours on $C:= c_1, c_2, \ldots, c_{|C|}$. In particular, we must have $|C|\geq \ell+1$.  Without loss of generality, $c_1$ is a neighbour of $v$. Furthermore, $c_i$ is a neighbour of $v$, for some $i\in \left\{\ell-1, \ell, \ldots, |C|-2\right\}$, since otherwise $v$ would have less than $\ell+1$ neighbours on $C$. But then $c_1,c_2,\ldots, c_i, v$ is a cycle of length at least $\ell$ and at most $|C|-1$, contradicting the minimality of $C$. 
We conclude that every vertex of $V(G)\backslash V(C)$ has at most $\ell$ neighbours in $V(C)$. Hence the minimum degree of $G-V(C)$ is at least $(k-1)\ell-1$ and so, by induction, $G-V(C)$ contains $k-1$ vertex-disjoint cycles of length at least $\ell$. Together with $C$, these cycles yield the desired subgraph of $G$.
\end{proof}

We remark that by the main result of~\cite{CLNWY17}, in fact the conclusion of Lemma~\ref{lem:disjointlongcycles} also holds for every graph with \emph{average} degree larger than $k(\ell+1)-2$, which is sharp.\\

The following auxiliary lemma will be applied in the main theorem to show that certain subgraphs of a minimum counterexample contain only short cycles.

\begin{lemma}\label{lem:diameter}
Let $\ell$ and $g \geq 2\ell-1$ be positive integers. Let $G$ be a multigraph with diameter $<g/2$ such that every cycle of $G$ has either length $< \ell$ or length $>g$. Then every cycle of $G$ in fact has length $< \ell$.
\end{lemma}
\begin{proof}
For a contradiction, assume that $G$ has a cycle of length at least $\ell$, and let $C$ denote such a cycle of minimum length. Note that $|C| > g$.
Consider two vertices $u,v \in V(C)$ that are at maximum distance  with respect to the subgraph $C$. Let $P$ be a shortest path in $G$ joining $u$ and $v$. Starting from $p_1:=u$, let $p_i$ denote the $i-$th vertex of $P$ that is on $C$. Observe that for all $i$, the subpath $P_i$ of $P$ joining $p_i$ and $p_{i+1}$ has no internal vertex on $C$.

Suppose that $P_i$ has strictly less than $d_C(p_i,p_{i+1})$ edges, for some $i\geq 1$. Then $P_{i} \cup C$ contains two cycles that are smaller than $C$, the largest of which has length at least $\frac{|C|}{2} > g/2 \geq \ell-\frac{1}{2} $. This contradicts the minimality of $C$. We conclude that $d_P(p_i,p_{i+1})\geq d_C(p_i,p_{i+1})$, for all $i\geq 1$.
 Combining this with the triangle inequality, we obtain $d_C(u,p_j) \leq \sum_{i=1}^{j-1} d_C(p_i,p_{i+1}) \leq \sum_{i=1}^{j-1} d_P(p_i,p_{i+1}) = d_P(u,p_j)$, for all $j\geq 2$. In particular, it follows that $g/2 \leq \lfloor \frac{|C|}{2} \rfloor = d_C(u,v) \leq d_P(u,v) \leq \text{diam}(G) <  g/2$; contradiction.
\end{proof}

For all integers $\ell\geq 2$ and  $k\geq 1$, we define $g(k,\ell):=8(\ell-1) \cdot (\lceil\log_2(k\ell)\rceil+1)$ and $g(0,\ell):=0$. We are now ready to prove our main theorem, a slightly refined version of Theorem~\ref{thm:longcycles}.

\begin{thm}\label{thm:main}
Let $\ell\geq 2$ be an integer. Call a cycle \emph{long} if it has length at least $\ell$.  Then for every multigraph $G$ and every nonnegative integer $k$, either $G$ contains $k$ edge-disjoint long cycles or there is a set $F\subseteq E(G)$ of size at most $k \cdot g(k,\ell)$ such that $G-F$ has no long cycle. 
\end{thm}
\begin{proof}
Throughout, we abbreviate $g(k):=g(k,\ell)$. 
Aiming for a contradiction, we assume the theorem is false; let $G$ be a counterexample minimizing $|E(G)|+|V(G)|$ and let $k\geq 1$ be such that $G$ does not contain $k$ edge-disjoint long cycles, nor a hitting set of size at most $k\cdot g(k)$. \\ 

Suppose first that $G$ contains a long cycle $C$ of length at most $g(k)$ and consider the graph $G'=G-E(C)$.  Since $G'$ is not a counterexample, it either contains $k-1$ edge-disjoint long cycles, or a set $F'\subseteq E(G')$ of size $|F'| \leq (k-1) \cdot g(k-1)$ such that $G'-F'$ has no long cycles. 
In the former case we can add $C$ to obtain $k$ edge-disjoint long cycles in $G$; contradiction. In the latter case, $F:=F' \cup E(C)$ is a hitting set for long cycles in $G$, while $F$ has size $|F'|+|E(C)| \leq (k-1) \cdot g(k-1) + g(k) \leq k\cdot g(k)$; contradiction. Therefore we may assume from now on that every long cycle in $G$ in fact has length larger than $g(k)$. Since a vertex of degree at most $1$ is not visited by any long cycle, we may also assume that the minimum degree of $G$ is at least $2$.\\
 
Next, we partition $V(G)$ into a collection of sets such that each of those sets induces a connected subgraph, as follows. Let $X$ be a maximal collection of vertices that are pairwise at distance at least $g(k)/4$. For each $x\in X$ and $s\in \mathbb{N}$, let $B_s(x)$ denote the set of vertices that are at distance at most $s$ from $x$. We will refer to $B_s(x)$ as the \emph{ball} of radius $s$ centered at $x$. Of particular interest are the balls of radius $r:= g(k)/8$. The balls $(B_r(x))_{x\in X}$ do not necessarily partition $V(G)$, but we will fix this by extending them. First add each vertex at distance $r+1$ of $X$ to one of the balls to which it is adjacent. After that, add every vertex at distance $r+2$ from $X$ to one of the sets constructed in the previous step to which it is adjacent. This process is continued until every vertex is covered. For each $x \in X$, we denote by $H(x)$ the graph induced by the final set that is obtained from $B_r(x)$ in this way. 

By construction, every vertex of $H(x)$ is at distance less than $g(k)/4$ from $x$; otherwise such a vertex could have been added to $X$, contradicting its maximality. Hence the diameter of $H(x)$ is less than $g(k)/2$ and so by Lemma~\ref{lem:diameter}, $H(x)$ contains no long cycle. This in turn implies that every block of $H(x)$ has diameter at most $(\ell-1)/2$.\\

We now analyse the block-tree structure of $H(x)$. First, we recall and introduce a few definitions. Given a graph $H$, a \emph{block} is a maximal connected subgraph of $H$ without a cut-vertex; in particular, an edge can be a block. A \emph{leaf-block} of $H$ is a block of $H$ that contains at most one cut-vertex of $H$.\\

Given a block $L$ of $H(x)$, an \emph{ancestor-block of $L$} is a block of $H(x)$ that is distinct from $L$ and which shares at least one edge with a shortest path between $V(L)$ and $x$.

For a correct analysis, we will also need to consider a subgraph $H^*(x)$ which consists of those blocks that are `not too close to the boundary of $H(x)$'. More formally, we define $H^*(x)$ to be the graph induced by the union of all blocks of $H(x)$ that have non-empty intersection with the ball $B_{r-\ell}(x)$. Note that each block of $H^*(x)$ is also a block of $H(x)$. Each leaf-block of $H^*(x)$ will be called a \emph{pre-leaf}.

The pre-leaves play an important role in our analysis, because these are the objects that we will actually count. 
A useful fact to keep in mind is that every vertex in a pre-leaf of $H(x)$ has no neighbour in $H(y)$ for any $y\in X, y\neq x$, since the pre-leaf has diameter at most $(\ell-1)/2$ and hence is contained in $B_{r- \lceil \frac{\ell-1}{2} \rceil}(x)$, while $H(y)$ does not intersect $B_{r}(x)$.\\

We will prove the following four claims, which together imply the theorem.

\begin{claim}\label{clm:0}
Let $x\in X$. Each pre-leaf is an ancestor-block of at least one leaf-block of $H(x)$.
 Conversely, if a leaf-block  $L$ of $H(x)$ has an ancestor-block which is a pre-leaf, then this pre-leaf is unique and we call it the \emph{pre-leaf of} $L$. 
\end{claim}

\begin{claim}\label{clm:1}
Let $x\in X$ and let $L$ be a leaf-block of $H(x)$.  Then there exists an $y\in X$ such that $L$ and $H(y)$ are joined by an edge of $G$.
\end{claim}

\begin{claim}\label{clm:2}
Let $x,y \in X$ be distinct and let $L_1,L_2$ be two leaf-blocks of $H(x)$ that have distinct pre-leaves. 
 Then at most one of $L_1, L_2$ is joined to $H(y)$ by an edge of $G$.
\end{claim}

\begin{claim}\label{clm:3}
Let $x\in X$. Then $H(x)$ has at least $k\ell-1$ distinct pre-leaves.
\end{claim}

Before proving these four claims, we first show how they imply the theorem. Let $x\in X$. By Claims~\ref{clm:0} and~\ref{clm:3}, $H(x)$ has at least $k\ell-1$ leaf-blocks that have pairwise distinct pre-leaves.  Therefore Claim~\ref{clm:1} and~\ref{clm:2} yield at least $k\ell-1$ distinct vertices $y \in X$ such that $H(x)$ is joined to $H(y)$ by an edge of $G$.  Now let $G^*$ denote the graph obtained from $G$ by contracting $H(x)$, for each $x\in X$. It follows that the minimum degree of $G^*$ is at least $k\ell-1$ and hence, by Lemma~\ref{lem:disjointlongcycles}, $G^*$ contains $k$ vertex-disjoint long cycles. This in turn implies that $G$ contains $k$ edge-disjoint long cycles; contradiction.\\

It remains to prove Claims~\ref{clm:0},~\ref{clm:1}, \ref{clm:2} and \ref{clm:3}. For that, the following two auxiliary claims are useful.
\begin{claim}\label{clm:noleafblock}
Let $x \in X$. No block of $H(x)$ is a leaf-block of $G$.
\end{claim}
\begin{claimproof}[Proof]
Suppose for a contradiction that some block $L$ of $H(x)$ is a leaf-block of $G$. Then every cycle in $G$ that visits an edge of $L$ must in fact be a subgraph of $L$. However, $L$ is a subgraph of $H(x)$ and thus contains no long cycle. We conclude that every long cycle of $G$ must be edge-disjoint from $L$. Hence the multigraph obtained from $G$ by contracting $L$ to a vertex must also be a counterexample, contradicting the minimality of $G$. 
\end{claimproof}

\begin{claim}\label{clm:consecutiveblocks}
Let $x\in X$ and let $T_1$ and $T_2$ be two blocks of $H(x)$ that both intersect the ball $B_{r-\ell}(x)$. Suppose furthermore that $V(T_1)\cap V(T_2)= \left\{v\right\}$ for some cut-vertex $v$ of $H(x)$ whose neighbours are in $V(T_1)\cup V(T_2)$. Then at least one of $T_1, T_2$ contains at least three cut-vertices of $H(x)$.
\end{claim}
\begin{claimproof}[Proof]
Suppose for a contradiction that $T_1, T_2$ both contain at most two cut-vertices of $H(x)$. First we deal with the case that one of them, say $T_1$, contains at most one cut-vertex of $H(x)$. Since $T_1$ has diameter at most $(\ell-1)/2$ and intersects $B_{r-\ell}(x)$, it follows that $T_1$ has no neighbours outside $H(x)$ and is thus a leaf-block of $G$. This contradicts Claim~\ref{clm:noleafblock}.

We may thus assume that both $T_1$ and $T_2$ in fact contain exactly two cut-vertices of $H(x)$, one of which must be their common vertex $v$. Let $t_i \in V(T_i)$ denote the cut-vertex of $H(x)$ which is not $v$, for each $i\in \left\{1,2\right\}$. Observe that $t_1\neq t_2$, for else $T_1$ and $T_2$ would not be blocks.

For $i\in \left\{1,2\right\}$, we write $p_i$ for the maximum number of edge-disjoint paths in $T_i$ between $t_i$ and $v$. Without loss of generality, $p_1\leq p_2$.
Next, let $G^*$ be the graph obtained from $G$ by contracting $T_2$ to a vertex (which we will call $v^*$). 
Note that $G^*$ has fewer edges than $G$, and so satisfies Theorem~\ref{thm:main}. This means that $G^*$ either contains $k$ edge-disjoint long cycles or a small hitting set. We will now analyse these two cases separately.\\

First suppose that $G^*$ contains a collection $\mathcal{C}$ of $k$ edge-disjoint long cycles. To obtain a contradiction, it suffices to derive that then $G$ also has $k$ edge-disjoint long cycles. To that end, we partition $\mathcal{C}$ into three subcollections $\mathcal{C}_1$, $\mathcal{C}_2$ and $\mathcal{C}_3$. We define $\mathcal{C}_1$ as the set of long cycles that visit at least one edge of $T_1$ (and thus in particular visit $v^*$), while $\mathcal{C}_2$ contains the long cycles that visit $v^*$ but are not already in $\mathcal{C}_1$ and, finally, $\mathcal{C}_3$ contains the long cycles that do not visit $v^*$. 
Decontracting $v^*$ back to $T_2$ yields the following natural bijection $\sigma$ from $\mathcal{C}$ to edge-disjoint cycles in $G$. If $C\in \mathcal{C}_3$, then the decontraction does not affect $C$, so we just take $\sigma(C)=C$. 
If $C \in \mathcal{C}_2$, then $\sigma(C)$ is the cycle obtained from $C$ by identifying $v^*$ with $t_2$ (in particular, this ensures that $\sigma(C)$ is edge-disjoint from $T_1$ and $T_2$). The last case is slightly more elaborate. Using that $T_1$ contains no long cycle, it follows that each $C\in \mathcal{C}_1$ must have a subpath that traverses $T_1$ from $t_1$ to $v^*$. Since there are at most $p_1$ edge-disjoint such paths, it follows that $|\mathcal{C}_1| \leq p_1 \leq p_2$. 
In turn, this implies that we can associate each $C\in \mathcal{C}_1$ with a unique subpath $P_C$  of $T_2$ that joins $v$ and $t_2$, such that these paths are edge-disjoint. Here we use that the neighbours of $v$ lie in $V(T_1) \cup V(T_2)$.
Now fix $C\in \mathcal{C}_1$. The decontraction of $v^*$  naturally maps $C$ to a path $P$ (on the same edges) which joins $v$ and $t_2$ in $G-E(T_2)$. Since $P$ and $P_C$ have the same endpoints and are internally vertex-disjoint, we can combine them into a cycle;  we define $\sigma(C)$ to be that cycle. This concludes the definition of the mapping $\sigma$. 
Note that for all $C$, the length of $\sigma(C)$ is at least the length of $C$, so we indeed obtain a collection of long cycles.  
Moreover, edge-disjointness is preserved because we did not modify the edge-sets outside $T_2$, while inside $T_2$ we added edge-disjoint paths.\\

Second, suppose that $G^*-F$ has no long cycles, for some $F\subseteq E(G^*)$ of size at most $k \cdot g(k)$. To obtain a contradiction, it suffices to derive that then $G-F$ also has no long cycles. (With slight abuse of notation, we refer to $F$ as a subset of both $E(G)$ and $E(G^*)$.)
Suppose for a contradiction that $G-F$ has a long cycle $C$. Then $C$ must use an edge of $T_2$.
Since $T_2$ has no long cycle, we know that $C$ consists of two internally disjoint paths $P_1, P_2$ that each join $v$ and $t_2$, where $P_1$ is a path in the graph $G-\left( V(T_2)\backslash\left\{ v, t_2 \right\}\right)$, while $P_2$ is a path in $T_2$. Here we use again that $v$ is only adjacent to vertices in $T_1$ and $T_2$.
Moreover, because $B_r(x)\subseteq H(x)$ has no long cycle, $P_1$ must pick up a vertex outside of $B_r(x)$.
Together with the fact that $V(T_2)$ is contained in $B_{r- l/2}(x)$, this implies that $P_1$ has length at least $\ell$. 
It follows that the cycle obtained from $C$ by contracting $P_2$ is a long cycle in $G^*$. Hence $F$ intersects $P_1$. But then $C$ cannot be a subgraph of $G-F$; contradiction. 
\end{claimproof}

At last, we are ready to prove Claims~\ref{clm:0},~\ref{clm:1}, \ref{clm:2} and \ref{clm:3}.

\begin{claimproof}[Proof of Claim~\ref{clm:0}]
Each pre-leaf is the ancestor-block of some leaf-block of $H(x)$, for otherwise $G$ would have a leaf-block, contradicting Claim~\ref{clm:noleafblock}.

The uniqueness of the pre-leaf of $L$ follows because of the following two reasons. First, given a shortest path $P$ between $V(L)$ and $x$, at most one leaf-block of $H^*(x)$ can share an edge with $P$. Second, any two shortest paths between $V(L)$ and $x$ must visit the same set of blocks of $H(x)$; otherwise there would be a cycle in $H(x)$ visiting at least two distinct blocks, contradicting the maximality of those blocks. 
\end{claimproof}

\begin{claimproof}[Proof of Claim~\ref{clm:1}]
If some leaf-block $L$ of $H(x)$ has no neighbour in $G-H(x)$, then $L$ is also a leaf-block of $G$, contradicting Claim~\ref{clm:noleafblock}.
\end{claimproof}

\begin{claimproof}[Proof of Claim~\ref{clm:2}]
Suppose for a contradiction that an edge $e_i$ joining $L_i$ with $H(y)$ would exist for both $i\in \left\{1,2\right\}$. A shortest path $P$ in $H(x)$ between $e_1$ and $e_2$ has length at least $\ell$. Indeed, $P$ must visit a common ancestor-block of $L_1$ and $L_2$. This common ancestor-block $C$ must intersect $B_{r-\ell}(x)$, since the pre-leaves of $L_1$ and $L_2$ are distinct. Together with the fact that the diameter of $C$ is at most $(\ell-1)/2$, it follows that the distance between $e_i$ and $C$ is at least $\frac{\ell}{2}$, for both $i\in\left\{1,2\right\}$. Furthermore, the length of $P$ is strictly less than $g(k)/2$, because every vertex of $H(x)$ is at distance less than $g(k)/4$ from $x$.
For the same reason, a shortest path $Q$ in $H(y)$ joining $e_1$ and $e_2$ also has length less than $g(k)/2$. Thus (noting that  $g(k)/2$ is an integer) the cycle $Pe_1Qe_2$ would be a long cycle of length at most $g(k)$; contradiction.
\end{claimproof}

\begin{claimproof}[Proof of Claim~\ref{clm:3}]
Fix $x\in X$. We define an auxiliary tree $T$ whose vertices are those blocks of $H(x)$ that intersect the ball $B_{r-\ell}(x)$. Two blocks $B_1, B_2$ are adjacent in $T$ if they have a common vertex and either $B_1$ is an ancestor-block of $B_2$ or vice versa. 

By construction, the leaves of $T$ represent the pre-leaves of $H(x)$, so we need to prove that $T$ has at least $k\ell$ leaves. From Claim~\ref{clm:noleafblock} and the fact that every block of $H(x)$ has diameter at most $(\ell-1)/2$, it follows that every leaf of $T$ is at distance at least $\frac{2}{\ell-1} \cdot (r-\ell)$ from the root. 
Moreover, by Claim~\ref{clm:consecutiveblocks}, $T$ does not have any adjacent degree-$2$ vertices that are within distance $\frac{2}{\ell-1} \cdot (r-\ell)$ from the root. 
It follows that $T$ has at least $2^{\frac{1}{2}\cdot \frac{2}{\ell-1}(r-\ell)} \geq k\ell$ leaves, as desired.
\end{claimproof}
\end{proof}

\section{Open problems}
\label{sec:open_problems}

As already mentioned in the introduction, the main open problem regarding the edge-EP property is to characterize the graphs $J$ such that $J$-expansions have the edge-EP property.

\begin{question}[\cite{RST16,BHJ18_counterexamples}]
Which graphs $J$ are such that $J$-expansions have the edge-EP property? 
\end{question}

It is known that such graphs $J$ need to be planar. 
Also, $J$ cannot be a ladder of length (number of rungs) at least $71$ nor a subcubic tree of pathwidth at least $19$, see~\cite{BHJ18_counterexamples}. 

Going back to cycles, another intriguing open problem is as follows. 
Thomassen~\cite{Thomassen1988} has proved that for any integer $m\geq 2$ and $p\in \left\{0,\ldots, m-1\right\}$, the class of cycles of length $p$ (mod $m$) has the vertex-EP property if and only if $p=0$. It is natural to expect that the same holds for the edge-EP property. As in the vertex case, the edge-EP property indeed does not hold when $p \neq 0$, as follows from the construction in~\cite{Thomassen1988}. For $m\geq 3$, it is unknown whether $0$ (mod $m$) cycles have the edge-EP property. 

\begin{question}
For $m\geq 3$, do cycles of length $0$ (mod $m$) have the edge-EP property? 
\end{question}

If they do, it is natural to expect that the edge-EP property then holds with an $O(k \log k)$ bounding function (where the hidden constant factor depends on $m$), as in the vertex case~\cite{CHJR19}. 

We conclude with a comment on the $m=2$ case. 
It is known that even cycles have the edge-EP property.  
This was shown recently by Bruhn, Heinlein, and Joos~\cite{BHJ18_framesApaths}, with an $O(k^2 \log k)$ bounding function. 
We note that the approach used in the current paper easily leads to an $O(k \log k)$ bounding function (the reader might want to try to show this for herself). 
However, an even quicker way is to derive it as a corollary from the fact that $\theta_3$-expansions have the edge-EP property with an $O(k\log k)$ bounding function~\cite{CRST18}, the key observation being that every $\theta_3$-expansion contains an even cycle as a subgraph.

\begin{thm}
\label{thm:evencycles}
Even cycles have the edge-EP property with an $O(k\log k)$ bounding function. 
\end{thm}
\begin{proof}
Every $\theta_3$-expansion contains two vertices joined by three internally vertex-disjoint paths, and at least two of these paths have either both odd length or both even length. Hence every $\theta_3$-expansion contains an even cycle.

Consider a graph $G$ and let $k\geq 1$. 
Suppose first that $G$ contains $k$ edge-disjoint $\theta_3$-expansions. Then $G$ also has $k$ edge-disjoint even cycles, as desired. Thus we may assume that $G$ does not contain $k$ edge-disjoint $\theta_3$-expansions and so by~\cite{CRST18}, there is a set $F$ of $O(k \log k)$ edges such that $G-F$ has no $\theta_3$-expansion. 
Note that every two cycles in $G-F$ are edge-disjoint; otherwise their union would contain a $\theta_3$-expansion. Let $t$ denote the number of even cycles in $G-F$. If $t\geq k$, then we find $k$ edge-disjoint even cycles in $G$, as desired, so we may assume that $t<k$. Selecting an edge from each even cycle of $G-F$, we obtain a set $F_2$ of $t$ edges such that $G-(F \cup F_2)$ has no even cycle. It remains to note that $|F \cup F_2| =O (k \log k + k)= O(k \log k)$.
\end{proof}

\section*{Acknowledgements}
We are grateful to Eunjung Kim, Micha\l{} Seweryn, and two anonymous referees for their comments on a previous version of the paper, which improved the paper.  

\bibliographystyle{abbrv}
\bibliography{edgeEP}
\end{document}